\documentclass[10pt|11pt|12pt]{article}
\usepackage{cite}
\usepackage{amssymb,amsmath,latexsym,mathrsfs}
\usepackage{psfrag}
\usepackage{bm}
\usepackage{cite}
\usepackage{url}
\usepackage{color}
\usepackage{graphicx}
\usepackage{caption}
\usepackage{subcaption}
\usepackage{booktabs}

\usepackage{cite}
\usepackage{amsthm}
\newtheorem{defi}{Definition}
\newtheorem{theorem}{Theorem}[section]
\newtheorem{lemma}[theorem]{Lemma}
\newtheorem{corollary}[theorem]{Corollary}

\newtheorem{example}{Example}

\newcommand{\fe}{  {K, \alpha, \boldsymbol{\beta}}  }
\newcommand{\nb}{   {\tiny \boxed{ \Psi}} }

\begin{document}

\title{\LARGE \bf FKG (and other  inequalities) via (generalized) FK representation (and
iterated folding) }

\author{ Alberto Gandolfi }
\vspace{5mm}


\maketitle

\thispagestyle{plain}
\pagestyle{plain}

\begin{abstract}

In this paper we prove several 
inequalities by
means of diagrammatic expansions,
a technique already used in \cite{BG13}. This time
we show that iterations of the folding of a probability  leads
to the proof of some inequalities by means of  a generalized
random cluster representation of the iterated foldings.
One of the inequalities is the well known FKG inequality,
which ends up being proven, quite unexpectedly,   by means of the (generalized) FK
representation. 

Although most of the results are not new, we hope that the techniques will 
find applications in  other contexts.

\end{abstract}


\section{Dedication}

This paper is dedicated to Chuck Newman on the occasion of his 
birthday.
The FKG inequality and the FK random cluster representations 
have been present in Chuck Newman research
in many works, starting from \cite{N80, N83, N94}, and, 
recently, in \cite{CJN17}. They appeared as two separate topics, as
they have always been so far.
They also played a role in many discussions I had with Chuck, during
which I gained  a deep insights on several problems in 
statistical mechanics, in particular spin glasses, and on many other subjects.

I present here a  proof of the FKG inequality which
 relies
on (an extension of) the FK representation (with some additional ingredients, such as foldings). Besides shedding some
light on the relation between the two concepts, I believe that the theory
developed here can be useful in other directions, in particular towards
identifying a role for percolation in the description of the phase transition in 
Spin Glasses (see \cite{NMS08}).
 One of the motivations behind this note is the hope that a close relation might, in due time, be found 
in the foldings.

\section*{Aknowledgments}
The author would like to thank J. van den Berg for very valuable discussions and suggestions.

\section{Introduction} 

\cite{BG13} introduces a technique which allows to prove some inequalities
of BK type concerning disjoint occurrences of events by means of 
diagrammatic expansions. The main tool there is to
consider foldings of a probability and then a generalized FK random
cluster representation of the folded measures. Depending on 
the features of  the random
cluster representation one could prove several different inequalities.

We consider here a further step, in which the folding operation is
iterated and  a tree of possible foldings is generated.
We develop  the terminology
and study the main properties.

We next show that it is possible to consider the random cluster representation of the 
distribution obtained after infinitely many foldings, and that specific representations 
can be obtained depending on the properties of the starting probability.
Such specific representations can then be used to prove inequalities.
The first examples involves the FKG inequality.

The FKG inequality was originally proven as a tool to study  models in 
rigorous statistical mechanics, such as the Ising model \cite{FKG71}.
Almost the same group of authors, and at about
the same time, developed, in an independent work, the FK random cluster representation,
as an alternative parametrization
of the Ising model \cite{FK72}. 
Both the FKG inequality and the FK representations have been 
generalized, and then used in an enormous number of situations (see \cite{G06},
for example).
Note, in particular, that the FK random cluster model representation of
the ferromagnetic Ising model satisfies the FKG inequality. In spite of these
connections, no close relation has ever been found between the
FK representation and the FKG inequality.
By means of the methods introduced here, however, we 
give here an alternative proof of the FKG inequality based on the (generalized
version of the) FK random cluster representation, a
quite   unexpected connection between the two concepts. Other inequalities
are also derived.

The strategy of the paper is as follows. 

First, we recap from \cite{BG13} how folding  and
generalized random cluster representation of the folded probabilities
allows to prove BK type inequalities. We then develop  an 
approximate version of the main results of \cite{BG13}.

Next, we  consider iteration of the folding operation giving rise to a tree of foldings, and note
that one can take the limit of infinitely many foldings. In addition,
we observe that one can restrict to the case in which essential
foldings, as defined below, are all done at the beginning, thus effectively reducing
the study of the limit distribution to a finite number of possible limits.

We then observe that the FKG theory, which
states that the FKG condition implies positive association,
 can be translated into the foldings where it takes an even more natural form.
A probability satisfies the FKG condition if and only if it does the same in 
each folded version, and hence in the limit of infinitely many foldings.
In the opposite direction, if the folded versions satisfy a certain inequality then
the initial probability has positive association. The above mentioned
inequality does indeed climb back the tree of foldings; the price to pay
in this is an exponential factor, which is however controlled by a super exponential
convergence of the distributions to their limits down the tree of foldings.

Starting from a probability $P$ which satisfies the FKG condition, the FKG condition
descends then down through the tree of foldings; it  is there, in the 
infinite limit on the tree of foldings,
 that the FKG condition meets
a simple FK random cluster representation able to easily justify
the inequality which, in turn, climbs back to prove positive association of $P$.

When all of this is applied to negative association it gives some sufficient conditions
for negative association, including Pemantle's result \cite{P00}. 

We hope that all this machinery, which is used here to give an alternative
proof to largely known results
may turn useful for novel results as well.
For simplicity, we consider only finite sets here,
leaving questions about infinite volume limits to further researches.

\section{Preliminaries}

\subsection{Definitions} 
Let $\Lambda$ be a finite set, and $\mathcal F = \{F_i\}_{i \in \Lambda}$
be a collection of finite sets $F_i$.
For each subset $M \subseteq \Lambda$
we denote $\Omega_{ M} = \prod_{i \in M} F_i$. 
We define the usual concatenation of configurations:
if $\omega(1)$ is defined on $B(1)$ and 
$\omega(2)$ in $B(2)$, with $B(1) \cap B(2)=\emptyset$,
then $\omega(1) \omega(2)$ is defined on $B(1) \cup B(2)$
such that $(\omega(1) \omega(2))_i = 
\omega(1)_i $ if $i \in B(1)$ and 
$(\omega(1) \omega(2))_i = 
\omega(2)_i $ if $i \in B(2)$.

For each set $S$, let 
$\mathcal P(S)$ be the family of all subsets of $S$;
we are interested in studying 
probabilities $P$ on $\mathcal P(\Omega_{\Lambda})$.

Throughout the paper $Z_i$ are normalizing constants implicitly defined
by the context.

When $J, K$ etc. represents entities like indices or sets, 
boldfaced characters ${\bf J}, {\bf K}$ etc. represent finite ordered collections of such entities
indexed in some form, i.e. vectors. In the sequel $\mathbb I_A$
represents the indicator function of $A$.

\subsection{Generalized FK-RCR}
A generalized FK random cluster representation, hereafter called RCR,
has been developed in \cite{BG13}, Section $2.1$, as follows.
\begin{defi}
Given finite sets $M$ and $ F_i, i \in M, \Omega_M=\prod_{i \in M} F_i$,
 a family ${\mathcal B} \subseteq {\mathcal P} (M)$; $H =
 \prod_{b \in \mathcal B} \mathcal P(\Omega_b)$;
  and a probability $P$ on 
$\mathcal P(\Omega_M)$,
a $\cal B$-RCR of $P$ is a probability $\nu$ on $\mathcal P(H)$
such that for every $\omega \in  \Omega_M$
\begin{eqnarray} \label{1}
P(\omega) = \frac{1}{Z_1} \sum_{\eta \in H} \nu(\eta) \prod_{ b \in \mathcal B}
\mathbb I_{\omega_b \in \eta_b} =  \frac{1}{Z_1} \sum_{\eta \in H, \eta \sim \omega} \nu(\eta) ,
\end{eqnarray}
where $Z_1$ is a normalizing factor, and $\eta \sim \omega$ means that
$\omega_b \in \eta_b$ for all $b \in \mathcal B$.
\end{defi}
We call $\mathcal B$ the collection of hyperbonds of the RCR, 
and $\nu$ its base probability or simply base.

\bigskip

As discussed in \cite{BG13}, our terminology is slightly different from that used
in \cite{FK72} and subsequent literature \cite{G06}. 
In the first place, we consider set variables $\eta$ instead of
the usual real valued ones. Then
 we focus on the base probability $\nu$ of the RCR; 
the works preceding  \cite{BG13}, instead, were focusing on the joint
distibution $Q(\omega, \eta)= 
 \frac{1}{Z_2}  \nu(\eta) \prod_{ b \in \mathcal B}
\mathbb I_{\omega \sim \eta}$, and, in particular, on the projection
on the $\eta$ variables $\phi(\eta)=\sum_{\omega} Q(\omega, \eta)$,
in which the well-known expression $2^{Cl(\eta)}$ appears.
The use of set variables and the focus on the base probability $\nu$ help streamlining the theory.

To illustrate the connection between the two notations, we now review the FK
representation for the 
Ising model without magnetic field. 
The collection of hyperbonds is made of the edges $E$ of some graph, i.e. $b$ such that $
b \in E$, and hence $ |b|=2$. For each $b=\{i,j\}$, in the standard FK
 we have the real valued
random variable $\mathbb I_{\omega_i = \omega_j}$, which corresponds to the 
set variable $\eta_b$ taking values either $\{(1,1),(-1,-1)\}=
\{\omega_{i,j}: \mathbb I_{\omega_i = \omega_j} \geq 1\}$,
or $\Omega_b=
\{\omega_{i,j}: \mathbb I_{\omega_i = \omega_j} \geq 0\}$;
hence $\nu$ is concentrated on
$ \prod_{b \in \mathcal B} H_b \subseteq H$ with
$H_b= \{\{(1,1),(-1,-1)\} , \Omega_b\}$.
Our $\nu$ is Bernoulli: $\nu=\prod_{b \in \mathcal B} \nu_b$,
where $\nu_b$ is defined on $H_b$ by
$\nu_b(\{(1,1),(-1,-1)\}) = p= 1-\nu_b( \Omega_b)$. We then get
$$
\nu(\eta) =  \frac{1}{Z_1}  p^{|\{b: \eta_b=\{(1,1),(-1,-1)\}\}|} 
(1-p)^{|\{b: \eta_b=\Omega \}|} =  \frac{1}{Z_1}  p^{n_1(\eta)} 
(1-p)^{n_0(\eta)} 
$$
where $n_i$ indicates the number of bonds $b$ with $|\eta_b|=2$. The joint representation is then
$Q(\omega, \eta)=  \frac{1}{Z_2}  p^{n_1(\eta)}
(1-p)^{n_0(\eta)} \mathbb I_{\omega \sim \eta}$; the marginal on $\omega$
is the Ising model with interaction $J$ such that $p=1-e^{-2 J}$,
and the marginal on $\eta$ is the original FK model
$$
\phi(\eta)= \frac{1}{Z_2}  p^{n_1(\eta)} 
(1-p)^{n_0(\eta)} 2^{Cl(\eta)}
$$
where $Cl(\eta)$ is the number of vertex clusters determined by the 
configuration $\eta$ by stating that two vertices $i,j$ are connected if
the $\mathbb I_{\omega_i = \omega_j} = 1$, i.e. 
$\eta_{i,j} \neq \Omega_b$ One pleasant feature of the general RCR is that a
probability $P$ is Gibbs iff it has a Bernoulli RCR (see \cite{BG13}).

All other extensions of the FK representation, such as the one in 
\cite{CL06}, can be easily translated in our framework.

\bigskip

A  notion of connectedness, similar to the one in the original work on random cluster 
representations, is developed in \cite{BG13} for
a general RCR: a hyperbond $b$ is called active in $\eta$ if 
 $\eta_b \neq \Omega_b$; and two sets of vertices
in $\Lambda$ are connected if there is a sequence of active hyperbonds
connecting them. This notion is central in proving the main inequalities
in \cite{BG13}.

\subsection{Foldings of a probability}
In \cite{R00, BG13, CS16} the operation of folding of a probability measure $P$
 is introduced, which amounts to the following: fix one region $K$, one configuration $\alpha$
 on $K$, two pointwise different configurations  $\boldsymbol{\beta}
 =(\beta(1) , \beta(2))$
 of $K^c$, and consider
 a random configuration of $K^c$ which pointwise takes  the value of
 either $\beta(1)$ or $ \beta(2)$, with the probability induced by $P\times P$
 given that both configurations on $K$ coincide with $\alpha$.
 This determines a collection of probabilities as $K, \alpha, \boldsymbol{\beta}$
 vary.
 As each of the resulting probabilities is symmetric (see below),
the entire process can be seen as a dependent coupling of $P$
with itself, which separates an asymmetric nonrandom  "drift" and a 
symmetric randomness.
 
 \begin{defi}
Given a probability $P$ on $\mathcal P(\Omega_{\Lambda})$, a subset 
$K \subseteq \Lambda$, a configuration $\alpha \in \Omega_K$,
and two configurations $\boldsymbol{\beta}  =(\beta(1) , \beta(2)) \in \Omega_{K^c}$, 
with $\beta_i(1) \neq  \beta_i(2)$ for all $i \in K^c$,
we define the $(K, \alpha, \boldsymbol{\beta})$-folded version of $P$
as the probability $P^{K, \alpha, \boldsymbol{\beta}}$ on 
$\Omega_{K^c}$ given by
\begin{eqnarray} \label{2}
P^{K, \alpha, \boldsymbol{\beta}}(\omega_{K^c})
=  \frac{1}{Z_{\fe} } P(\alpha \omega_{K^c})P(\alpha \overline \omega_{ K^c}^{\boldsymbol{\beta}}),
\end{eqnarray}
if there is at least one $\omega_{K^c}$ for which the r.h.s. is not zero; in the definition
 we used concatenation of configurations, and the 
reverse operation $ \overline \omega^{\boldsymbol{\beta}}$ which consists
of exchanging $\omega_i$ from $\beta_i(1)$ to $\beta_i(2)$ or viceversa,
if $\omega_i \in \{\beta_i(1), \beta_i(2)\}$ and inserting a symbol with zero $P$ 
probability otherwise.
Clearly, $P^{K, \alpha, \boldsymbol{\beta}}(\omega_{K^c})$ is concentrated on the
$\omega_{K^c} \in \prod_{i \in K^c} \{\beta_i(1), \beta_i(2)\}$. 
\end{defi}
An equivalent definition is obtained by considering
$$ W_{\fe} = \{(\omega, \omega') \in \Omega\times \Omega :
\omega_K =\omega'_k=\alpha, \text{and } \omega_i,
\omega'_i \in \{\beta_i(1), \beta_i(2)\} \text{ for all } i \in K^c \},
$$
then  taking $P^{K, \alpha, \boldsymbol{\beta}}(\omega_{K^c})
= (P\times P) (\cdot | W_{\fe})$, and finally projecting on the first configuration.
This entails that $Z_{\fe} = (P\times P) (W_{\fe})$.

When there are only two symbols, i.e. $|F_i|=2$, then 
$\boldsymbol{\beta}$ are irrelevant and can be omitted from the
notation.

\begin{example}
The folding of an Ising model $\mu_{{\bf J}, {\bf h}}(\omega)
=\frac{1}{Z} \exp ( \sum_{i,j} J_{i,j} \omega_i \omega_j + \sum_i h_i \omega_i)$,
$\omega \in \{-1,1\}^{\Lambda}$,
is an Ising model $P^{K, \alpha, \boldsymbol{\beta}}=
P^{K, \alpha}= \mu_{2{\bf J}, {\bf 0}}$ with twice the interaction and
zero external field.
\end{example}

\subsection{Disjoint occurrences of events}
\cite{BK85} introduces the concept of disjoint occurrence of events,
proving the first version of the BK inequality, shown then in full 
generality in \cite{R00}.  \cite{BG13}  introduces several different
versions of disjoint occurrence, using the RCR of the folding 
of a probability $P$ to show that the required inequalities
hold for certain $P$'s. 

In all these works, an important role is played by a 
function which indicates which pairs of sets of indices can be used
to identify two events in a given configuration. More precisely,
for $A, B \subseteq \Omega_{\Lambda}$, the set of disjoint 
occurrence pairs is defined by
$$
\mathcal D(A, B, \omega)= \{(K,L):K, L \subseteq \Lambda,
K \cap L= \emptyset, [\omega]_K \subseteq A,  [\omega]_L \subseteq
B\}.
$$
A selection rule is a function $\Psi$ which assigns to each 
$(A, B, \omega)$ a (possibly empty) subset 
$\Psi(A, B, \omega)$ of $\mathcal D(A, B, \omega)$,
i.e. a set of disjoint pairs of subsets of $\Lambda$ from which $A$ and $B$ can be 
disjointly recognized. By means of a
selection rule we can define the generalized box 
operation   \cite{BG13} 
\begin{eqnarray} \label{GenBox}
A \nb  B= \{ \omega \in \Omega_{\Lambda}:
\Psi(A, B, \omega) \neq \emptyset \}.
\end{eqnarray}
If $\Psi(A, B, \omega)=\mathcal D(A, B, \omega)$ we get the usual
box operation as $A   {\tiny \boxed{\mathcal D}} B = A \Box B$.  
\begin{example} \label{SelectionForIncreasing}
If  $F_i=\{0,1\}$ for all $i \in \Lambda$
and $A$ and $B$ are increasing, then 
we can require that $(K,L) \in \Psi(A, B, \omega) $
implies $K, L \subseteq \omega^{-1}(1)$
to disallow the possibility of recognizing events in sets which
uselessly contain vertices $i$ in which $\omega_i=0$.
\end{example}
\begin{example} \label{SelectionForIncreasingDecreasing}
Similarly, if  $F_i=\{0,1\}$ for all $i \in \Lambda$, $A$ is  increasing
and $B$ is decreasing, then we can take 
$K \subseteq \omega^{-1}(1)$
and 
$ L \subseteq \omega^{-1}(0)$.
\end{example}
In this paper we will eventually be dealing only with these last two cases, but 
we state the results for a general $\Psi$ as it does not require much more efforts.

\subsection{Inequalities from approximate RCR and folding}
We 
start by extending the main result of \cite{BG13}  to the case in which the RCR 
of the folding 
is only approximate. 
\begin{defi}
Given finite sets $M$ and $ \Lambda, M \subseteq \Lambda, F_i, i \in M, \Omega_M=\prod_{i \in M} F_i$,
 a family $\cal {B} \subseteq \cal {P}$$( \Lambda) $, constraints $H \subseteq 
 \prod_{b \in \mathcal B'} \mathcal P(\Omega_b)$
 on hyperbonds configurations, with $\mathcal B'=\{b'=b \cap M, b \in \mathcal B\}$,  a probability $P$ on 
$\mathcal P(\Omega_M)$, and $\epsilon >0$,
a {\bf $\epsilon$-$\cal B$-RCR} of $P$ is a probability $\nu$ on $\mathcal P(H)$
such that for every $\omega \in  \Omega_M$
\begin{eqnarray} \label{1.1}
P(\omega) =  \frac{1}{Z_1} \sum_{\eta \in H, \eta \sim \omega} \nu(\eta)
\pm \epsilon,
\end{eqnarray}
where $x=y \pm \epsilon$ stands for $|x-y| \leq \epsilon$.

\end{defi}

\begin{lemma} \label{2.3.2}
Let $\Lambda$ be a finite set, $\Omega_{\Lambda} = \prod_{i \in \Lambda} F_i$,
and suppose $|F_i|=2$ for all $i \in \Lambda$; let $P$ be a probability on 
$\Omega_{\Lambda}$, $A,B \subseteq \Omega_{\Lambda}$,
$\Psi$ a selection rule. Suppose that, for $\epsilon >0$, $P$ 
has a symmetric $\frac{\epsilon}{2^{|\Lambda |+1}}$-${\mathcal B}$-RCR
$\nu$ such that $\Psi$ uses disjoint random clusters of $\nu$ for 
$A,B$, then 
$$
P(A \nb B) \leq P(A \cap \overline B) \pm   \epsilon,
$$
where $\overline B$ is the event obtained from $B$ by 
changing all the $\omega_i$'s into $\overline \omega_i=
1-\omega_i$. Here symmetric means that $\nu(\prod_{b \in \mathcal B}
\eta_b) = \nu(\prod_{b \in \mathcal B}\overline \eta_b)$, and connectedness and related
clusters
are defined at the end of  Section 3.2.

\end{lemma}
\begin{proof}  The proof of a similar statement is given in 
\cite{BG13}, so we omit some tedious details here.
Given a configuration $\eta$ the set of active hyperbonds $b$ falls apart into
$t(\eta)$ connected components $C_1(\eta), \dots, C_{t(j)}(\eta)$ called clusters. Fix now a configuration $\omega$
such that $\omega_b \in \eta_b$ for each $b \subseteq \cup_{I=1}^{t(\eta}
C_{\ell}(\eta)$. 
By the symmetry of $\nu$ and the definition of active hyperbonds, with $\nu$-probability one
 the configuration $\overline \omega^{C_{\ell}(\eta)}$,
which is the reversed $1-\omega_i$ for all $i \in {C_{\ell}(\eta)}$
and coincides with $\omega_i$ for all other $i$'s, also satisfies 
$\overline \omega^{C_{\ell}(\eta)}_b \in \eta_b$ for all $b$'s. Let then 
$\Omega_{\eta, \omega}$
be the set of configurations $\xi \in \Omega_{\Lambda}$ which  coincide with
either $\omega$ or $\overline \omega^{C_{\ell}(\eta)}$ for each $\ell=1,\dots,t(\eta)$; we have
 $|\Omega_{\eta, \omega}|=2^{t(\eta)}$. Consider the joint probability distribution
 $Q(\rho, \xi)=\frac{1}{Z_2} \nu(\rho) \prod_{b \in \mathcal B} \mathbb I_{\xi
 \sim \rho}$; then for the given $\eta$ and $\omega$,
 $\mu_{\eta,\omega}( \xi) = Q( (\rho, \xi) | \rho=\eta, \xi \in \Omega_{\eta, \omega} )
 =\frac{ \mathbb I_{\xi \in \Omega_{\eta, \omega} }}{2^{t(\eta)}}$
 is concentrated on $\Omega_{\eta, \omega}$. This 
 distribution can then be encoded into i.i.d. two valued symmetric Bernoulli variables
 by a map $T_{\eta}:\Omega_{\eta, \omega} \rightarrow \{0,1\}^{t(\eta)}$ with
 $(T_{\eta} (\xi))_{\ell}=
  \mathbb I_{\xi |_{C_{\ell}(\eta)} = \omega |_{C_{\ell}(\eta)} } $.
  Reimer's \cite{R00} results apply to $T_{\eta}(\mu_{\eta})$ 
   showing that $T_{\eta}(\mu_{\eta})(A' \Box B')
  \leq T_{\eta}(\mu_{\eta})(A' \cap \overline B'$ for all
  events $A', B' \subseteq \{0,1\}^{t(\eta)}$ .
  
  In addition, one can verify, as it is explicitly proven in \cite{BG13},  see the claims on Page 171, 
  that $T_{\eta}( A \nb B) \subseteq T_{\eta}( A ) \Box T_{\eta}( B)$,
  and $T_{\eta}( A \cap B) = T_{\eta}( A ) \cap T_{\eta}( B)$.
  The relation $R(\xi, \omega)$ iff $\xi \in \Omega_{\eta, \omega} $
  partitions$ \{ \omega \in \Omega_{\Lambda}: \omega
\sim \eta\}$ into equivalence classes; indicating
 $\Omega(\eta)=\{ \omega \in \Omega_{\Lambda}: \omega
\sim \eta\}/R$, we have
\begin{eqnarray} 
P( A \nb B) &= & \sum_{\omega \in  A \nb B} P(\omega) \nonumber\\
&=& \sum_{\omega \in  A \nb B}   \frac{1}{Z_1} \sum_{\eta \in H} \nu(\eta)
\mathbb I_{ \eta \sim \omega} \pm 2^{|\Lambda|} \frac{\epsilon}{2^{|\Lambda |+1}}  \nonumber  \\ 
&=&  \sum_{\eta \in H} Q(\eta,  A \nb B)  \pm 2^{|\Lambda|} \frac{\epsilon}{2^{|\Lambda |+1}}  \nonumber \\ 
&=&
 \sum_{\eta \in H} \sum_{\omega \in \Omega(\eta)} \mu_{\eta,\omega}(A \nb B) Q(\Omega_{\eta, \omega}) \pm 2^{|\Lambda|} \frac{\epsilon}{2^{|\Lambda |+1}}  \nonumber \\
&\leq&  \sum_{\eta \in H} \sum_{\omega \in \Omega(\eta)} T_{\eta} (\mu_{\eta,\omega})(T_{\eta}(A) \Box T_{\eta}( B)) Q(\Omega_{\eta, \omega}) \pm 2^{|\Lambda|} \frac{\epsilon}{2^{|\Lambda |+1}} \\
&\leq& \sum_{\eta \in H} \sum_{\omega \in\Omega(\eta)} T_{\eta} (\mu_{\eta,\omega})(T_{\eta}(A) \cap \overline{ T_{\eta}( B)}) Q(\Omega_{\eta, \omega}) \pm 2^{|\Lambda|} \frac{\epsilon}{2^{|\Lambda |+1}}  \nonumber\\
&=&\sum_{\eta \in H} \sum_{\omega \in\Omega(\eta)} T_{\eta} (\mu_{\eta,\omega})(T_{\eta}(A \cap \overline{  B}) Q(\Omega_{\eta, \omega}) \pm 2^{|\Lambda|} \frac{\epsilon}{2^{|\Lambda |+1}}  \nonumber\\
&=&\sum_{\eta \in H} \sum_{\omega \in \Omega(\eta)} \mu_{\eta,\omega}(A \cap 
\overline B) Q(\Omega_{\eta, \omega}) \pm 2^{|\Lambda|} \frac{\epsilon}{2^{|\Lambda |+1}}  \nonumber\\
&=&P(A \cap \overline B)  \pm  \epsilon \nonumber
\end{eqnarray}

\end{proof}
  
Now define the functional
$$
R_{K, \alpha, \boldsymbol{\beta}}(\omega)
=  \{  \omega_{K^c} : \alpha  \omega_{K^c} = \omega
 \text{ and }
( \omega_{K^c})_i \in  \{\beta_i(1), \beta_i(2)\}   \text{  for all }  i \in K^c \}.
 $$
 
If $\Psi$ is a selection rule on $\Lambda$, $K \subseteq \Lambda$
and $\alpha \in \Omega_K$,
  we let $\Psi_{K, \alpha}$
be defined by
$\Psi_{K, \alpha}(A', B',  \omega_{K^c})
= \Psi(\alpha A', \alpha B', \alpha \omega_{K^c})\cap K^c$ for all $A', B' \subseteq 
\Omega_{K^c}$,
where $\alpha A'$ is the concatenation of the configurations of $A'$
with $\alpha$,  and for a collection of pairs  of sets $(M,N)$'s
the intersection with $K^c$ is simply the collection of intersections
$(M \cap K^c,N \cap K^c)$.
 $\Psi_{K, \alpha}$
is a selection rule in $\mathcal P(\Omega_{K^c})\times \mathcal P(\Omega_{K^c})
\times \Omega_{K^c}$, and one can see that 
$$
R_{K, \alpha, \boldsymbol{\beta}}(A \nb B)
\subseteq R_{K, \alpha, \boldsymbol{\beta}}(A)
 {\tiny \boxed{ \Psi_{K, \alpha} }} R_{K, \alpha, \boldsymbol{\beta}}(B).
 $$

\begin{lemma} \label{2.3.3}
For  $\Lambda, F_i, \Omega_{\Lambda}, P, \Psi $ $A, B \in \Omega_{\Lambda}$,
 as before, we have that if for each $K, \alpha, \boldsymbol{\beta}$
\begin{eqnarray*}
 && P^{K, \alpha, \boldsymbol{\beta}} (R_{K, \alpha, \boldsymbol{\beta}}(A)
 {\tiny \boxed{ \Psi_{K, \alpha} }} R_{K, \alpha, \boldsymbol{\beta}}(B)) \\
&& \quad \leq  P^{K, \alpha, \boldsymbol{\beta}}(R_{K, \alpha, \boldsymbol{\beta}}(A)
 \cap \overline{ R_{K, \alpha, \boldsymbol{\beta}}(B) } ) \pm \epsilon
\end{eqnarray*}
then
$$
P(A \nb B) \leq P(A) P(B)  \pm \epsilon.
$$

\end{lemma}
\begin{proof} Recall that the 
$ W_{\fe}$'s, with varying $K \subseteq \Lambda,
\alpha \in \Omega_K, \boldsymbol{\beta} \in \Omega_{K^c}$
 form a partition of $\Omega_{\Lambda} \times \Omega_{\Lambda}$.
 Furthermore, $P^{K, \alpha, \boldsymbol{\beta}}(\omega_{K^c})
 =(P \times P)(\alpha \omega_{K^c} \times \Omega_{\Lambda}
 |W_{\fe})$
 so that for $A \subseteq \Omega_{\Lambda}$, 
 $$
 (P \times P)(A \times \Omega_{\Lambda}
 |W_{\fe})= P^{K, \alpha, \boldsymbol{\beta}}(R_{K, \alpha, \boldsymbol{\beta}}(A )).
 $$
 Thus 
\begin{eqnarray} 
P( A \nb B) &= & 
(P \times P)( A \nb B \times \Omega_{\Lambda}) \nonumber \\
 &= &  \sum_{K \subseteq \Lambda} \sum_{\alpha \in \Omega_K} \sum_{\boldsymbol{\beta}}
(P \times P)( A \nb B \times \Omega_{\Lambda}|W_{\fe}) \nonumber \\
 &&\quad  \quad  \quad  \quad \quad  \quad  \quad  \quad  \quad  \quad 
 \times  (P \times P)(W_{\fe})  \\
 &\leq &  \sum_{K \subseteq \Lambda} \sum_{\alpha \in \Omega_K} \sum_{\boldsymbol{\beta}}
((P \times P)(A
\times B|W_{\fe})\pm \epsilon) \nonumber \\
 &&\quad  \quad  \quad  \quad \quad  \quad  \quad  \quad  \quad  \quad 
  \quad  \quad  \quad \quad  \quad  \quad  
 \times(P \times P)(W_{\fe}) \nonumber \\
 &= & P(A) P(B) \pm \epsilon
  \nonumber 
\end{eqnarray}
as for each $K \subseteq \Lambda,
(P \times P)( A \nb B \times \Omega_{\Lambda}|W_{\fe}),
 \alpha \in \Omega_K, \boldsymbol{\beta} \in \Omega_{K^c}$
 we have
\begin{eqnarray} 
&&(P \times P)( A \nb B \times \Omega_{\Lambda}|W_{\fe})  \nonumber \\
 && \quad  \quad  \quad =
 P^{\fe}( R_{\fe}(A \nb B)) (P \times P)(W_{\fe})\nonumber \\
&&\quad  \quad  \quad \leq 
 P^{K, \alpha, \boldsymbol{\beta}} (R_{K, \alpha, \boldsymbol{\beta}}(A)
 {\tiny \boxed{ \Psi_{K, \alpha} }} R_{K, \alpha, \boldsymbol{\beta}}(B))
  \nonumber \\
 &&\quad  \quad  \quad  \quad \quad  \quad  \quad  \quad  \quad  \quad 
  \quad  \quad  \quad \quad  \quad  \quad  
 \times (P \times P)(W_{\fe})
 \nonumber \\
&&\quad  \quad  \quad \leq   
(P^{K, \alpha, \boldsymbol{\beta}} (R_{K, \alpha, \boldsymbol{\beta}}(A)
\cap \overline{ R_{K, \alpha, \boldsymbol{\beta}}(B)}) \pm \epsilon)
 \\
 &&\quad  \quad  \quad  \quad \quad  \quad  \quad  \quad  \quad  \quad 
  \quad  \quad  \quad \quad  \quad  \quad  \times (P \times P)(W_{\fe})
\nonumber \\
&&\quad  \quad  \quad=  (
(P \times P)(R_{K, \alpha, \boldsymbol{\beta}}(A)
\times { R_{K, \alpha, \boldsymbol{\beta}}(B)}) \pm \epsilon)
\nonumber \\
 &&\quad  \quad  \quad  \quad \quad  \quad  \quad  \quad  \quad  \quad 
  \quad  \quad  \quad \quad  \quad  \quad  \times (P \times P)(W_{\fe})
\nonumber \\
&&\quad  \quad  \quad=   (
  (P \times P)(A
\times B|W_{\fe})\pm \epsilon
)(P \times P)(W_{\fe}) \nonumber
\end{eqnarray}

\end{proof}
Combining results equivalent to Lemmas  \ref{2.3.2} and \ref{2.3.3}
for $\epsilon =0$, it is proven in \cite{BG13} that if $\Psi$ used
disjoint clusters of each RCR of the foldings of a probability $P$, 
then $P(A \nb B) \leq P(A) P(B)$.

\section{Iterated foldings}

\subsection{Iterated foldings}
We now introduce the notation for dealing with iterated foldings. After the first 
folding the distributions are binary and symmetric, so the first folding is often
singled out in the notation, and results for subsequent foldings assume symmetry.

We use the usual symbols $\Lambda, \mathcal B(\Lambda),
F_i, P$ on 
$\Omega_{\Lambda}$.
\begin{defi}
For $i=1,\dots,n$, let $K_i \subseteq \left( \cup_{j=1}^{i-1} K_j \right)^c$,
$K_j \subseteq \Lambda$, be a sequence of sets, and indicate 
${\bf K}_i=(K_1, \dots, K_i)$; let $\alpha_1 \in \Omega_{K_1}$, and for 
$i \geq 2$, $\alpha_i \in \Omega_{K_i}(\boldsymbol{\beta})$
and $\boldsymbol{ \alpha}_i=(\alpha_1, \dots, \alpha_i)$; and, finally,
let $\boldsymbol{\beta} =(\beta(1),\beta(2))$ with $\beta(i) \in  \Omega_{K_1^c}$.

Given $\mathcal B(\Lambda)$,
the $({\bf K}_n, \boldsymbol{ \alpha}_n, \boldsymbol{\beta})$-$\mathcal B$-folding 
$ P^{{\bf K}_n, \boldsymbol{ \alpha}_n, \boldsymbol{\beta}}$ of $P$
is defined recursively on $\Omega_{(\cup_{K \in {\bf K}_i} K)^c}$
by $ P^{{\bf K}_i, \boldsymbol{ \alpha}_i, \boldsymbol{\beta}}
:={(   P^{{\bf K}_{i-1}, \boldsymbol{ \alpha}_{i-1}, \boldsymbol{\beta}}  )}^{k_i, \alpha_i}$
for $i \geq 2$, starting from $ P^{{\bf K}_1, \boldsymbol{ \alpha}_1, \boldsymbol{\beta}}
:= P^{{ K}_1, { \alpha}_1, \boldsymbol{\beta}}$. It is natural to denote 
$ {\bf P}^{\bf 0} = P$, with $({\bf K}_0, \boldsymbol{ \alpha}_0, \boldsymbol{\beta})={\bf 0}$.

\end{defi}
The set of iterated foldings has the structure of an infinite tree, indexed
by 
$({\bf K}_n, \boldsymbol{ \alpha}_n, \boldsymbol{\beta})$, $n=0, 1, \dots$. The direct descendants
of $({\bf K}_i, \boldsymbol{ \alpha}_i, \boldsymbol{\beta})$ are of the form 
$(({\bf K}_i, K_{i+1}), (\boldsymbol{ \alpha}_i, \alpha_{i+1}), \boldsymbol{\beta})$
with $K_{i+1} \subseteq (\cup_{K \in {\bf K}_i} K)^c, \alpha_{i+1} \in \Omega_{K_{i+1}}$;
$ \boldsymbol{\beta}$ is not relevant after the first folding. Moreover, if
$K_i = \emptyset$ then $\alpha_i  \in \Omega_{K_i}$ is just a formal symbol (sometimes we use
a generic symbol $\alpha$ for this).

\bigskip

An {\bf infinite branch } in a tree of foldings is a sequence 
$({\bf K}_i, \boldsymbol{ \alpha}_i, \boldsymbol{\beta})$ of vertices of the tree. The 
sequence $(\cup_{K \in {\bf K}_i} K)^c$ is non increasing, and in an infinite
branch of the tree of foldings it must be asymptotically constant.
Thus, for any infinite branch there exists $L$ such that $K_i=\emptyset$ for
all $i \geq L$. 

The {\bf limiting distribution in an infinite branch} of the tree of foldings such that
$K_i=\emptyset$ for
all $i \geq L$ is the probability $P^{\infty}$ defined on $\overline \Omega_{\Lambda}
:= \Omega_{(\cup_{K \in {\bf K}_L} K)^c}$ by $P^{\infty}(\omega)
= \lim_{i \rightarrow \infty} P^{{\bf K}_i, \boldsymbol{ \alpha}_i, \boldsymbol{\beta}}(\omega)$.
\begin{lemma} \label{5.1.3}
For $L \geq 1$ the limiting distribution in an infinite branch of the tree of foldings
such that $K_i=\emptyset$ for
all $i \geq L$ always exists, and equals
$$
P^{\infty}(\omega)=\frac{1}{Z_{\infty}} \mathbb I_{\{\omega: P^{{\bf K}_L, \boldsymbol{ \alpha}_i, \boldsymbol{\beta}}(\omega)
\geq P^{{\bf K}_L, \boldsymbol{ \alpha}_i, \boldsymbol{\beta}}(\omega')
\text{ for all } \omega' \in \Omega_{{(\cup_{K \in {\bf K}_L} K)^c}} \} },
$$
i.e. it is the uniform distribution on the maxima of 
$ P^{{\bf K}_L, \boldsymbol{ \alpha}_L, \boldsymbol{\beta}}$,
where $Z_{\infty}$ is the number of such maxima. Moreover, for 
$i > L$, and any $\omega$ such that $P^{\infty}(\omega)>0$
\begin{eqnarray} \label{5.1.5}
&& \sup_{\omega \in \Omega_{L   }}
|P^{{\bf K}_i, \boldsymbol{ \alpha}_i, \boldsymbol{\beta}}(\omega) - 
P^{\infty}(\omega)| \\
&&\quad \quad \quad \leq   |\Omega_{\Lambda}| \frac{ \max_{\omega'': P^{\infty}(\omega')=0} 
(P^{{\bf K}_L, \boldsymbol{ \alpha}_L, \boldsymbol{\beta}}(\omega'') )^{2^{i-L}} }
{(P^{{\bf K}_L, \boldsymbol{ \alpha}_L, \boldsymbol{\beta}}(\omega') )^{2^{i-L}}} =:a^{2^i} \nonumber
\end{eqnarray}
with $a<1$.

\end{lemma}
\begin{proof}
let $i>L; since K_i =\emptyset$  all $P^{{\bf K}_i, \boldsymbol{ \alpha}_i, \boldsymbol{\beta}}$'s
are defined on the same $\Omega_{\Lambda}$; each direct descendant is a folding of the
parent distribution, which is already symmetric since $L \geq 1$.
Thus $P^{{\bf K}_{i+1}, \boldsymbol{ \alpha}_{i+1}, \boldsymbol{\beta}}(\omega)
= \frac{(P^{{\bf K}_i, \boldsymbol{ \alpha}_i, \boldsymbol{\beta}}(\omega) )^2}
{\sum_{\omega'} P^{{\bf K}_i, \boldsymbol{ \alpha}_i, \boldsymbol{\beta}}(\omega') )^2 }$.
Hence, if $\omega$ is a point of maximum for $ P^{{\bf K}_L, \boldsymbol{ \alpha}_L, \boldsymbol{\beta}}$, then 
$$
\lim_{i \rightarrow \infty} P^{{\bf K}_i, \boldsymbol{ \alpha}_i, \boldsymbol{\beta}}(\omega) 
= \lim_{i \rightarrow \infty}  
 \frac{(P^{{\bf K}_L, \boldsymbol{ \alpha}_L, \boldsymbol{\beta}}(\omega) )^{2^{i-L}}}
{\sum_{\omega'} (P^{{\bf K}_L, \boldsymbol{ \alpha}_L, \boldsymbol{\beta}}(\omega') )^{2^{i-L}} }
=\frac{1}{Z_{\infty}},
$$
and the limit equals $0$ otherwise. Moreover, for all $\omega$
such that $P^{\infty}(\omega)>0$, as $Z_{\infty}\geq 1$, we have
\begin{eqnarray} 
&&|P^{{\bf K}_i, \boldsymbol{ \alpha}_i, \boldsymbol{\beta}}(\omega) - 
\frac{1}{Z_{\infty}}| \nonumber \\
&& \quad \quad =
| \frac{(P^{{\bf K}_L, \boldsymbol{ \alpha}_L, \boldsymbol{\beta}}(\omega) )^{2^{i-L}}}
{Z_{\infty} (P^{{\bf K}_L, \boldsymbol{ \alpha}_L, \boldsymbol{\beta}}(\omega) )^{2^{i-L}}+
\sum_{\omega'':P^{\infty}(\omega'')=0} (P^{{\bf K}_L, \boldsymbol{ \alpha}_L, \boldsymbol{\beta}}(\omega'') )^{2^{i-L}} }- 
\frac{1}{Z_{\infty}}|
 \nonumber\\
&& \quad \quad \leq \frac{ 
\sum_{\omega'':P^{\infty}(\omega'')=0} (P^{{\bf K}_L, \boldsymbol{ \alpha}_L, \boldsymbol{\beta}}(\omega'') )^{2^{i-L}} }{ (P^{{\bf K}_L, \boldsymbol{ \alpha}_L, \boldsymbol{\beta}}(\omega) )^{2^{i-L}}}
 \nonumber\\
&& \quad \quad \leq |\Omega_{\Lambda}| \frac{ \max_{\omega'': P^{\infty}(\omega'')=0} 
(P^{{\bf K}_L, \boldsymbol{ \alpha}_L, \boldsymbol{\beta}}(\omega'') )^{2^{i-L}} }
{(P^{{\bf K}_L, \boldsymbol{ \alpha}_L, \boldsymbol{\beta}}(\omega) )^{2^{i-L}}} .\nonumber
\end{eqnarray}
As $P^{{\bf K}_L, \boldsymbol{ \alpha}_L, \boldsymbol{\beta}}(\omega)
> P^{{\bf K}_L, \boldsymbol{ \alpha}_L, \boldsymbol{\beta}}(\omega'')$ for all
the finite number of $\omega''$'s we can take $a<1$.

\end{proof}

\subsection{Essential tree of foldings} 
An {\bf essential folding} is either a folding in which $K_i \neq \emptyset$
or the first folding (which  is "essential" by itself as it simmetrizes the distribution);
the other foldings are inessential.
The {\bf essential tree of  foldings} is a subtree of the tree of foldings such that
in each branch the first foldings are all essential and the remaining ones are
all inessential; in other words, it does not alternate between the two types of foldings.

We now see that the essential tree of foldings contains all the information of the
tree of foldings.
\begin{lemma}\label{essential foldings}
If $  P^{{\bf K}_{n}, \boldsymbol{ \alpha}_{n}, \boldsymbol{\beta}} $
is a folding then there is a folding in the generation of the essential tree of foldings
equal to it.
\end{lemma}
\begin{proof}
If ${\bf K}_{n}=(K_1, K_2, \dots, K_n)$ then let $K_{i_1}, K_{i_2}, \dots, K_{i_r}$
be the nonempty $K_i$'s among those with $i\geq 2$. We see that  for 
${\bf K'}_{n}=(K_1, K_{i_1}, K_{i_2}, \dots, K_{i_r},\emptyset, \dots, \emptyset)$,
and $ \boldsymbol{ \alpha'}_{n}$ similarly rearranged,
 $  P^{{\bf K}_{n}, \boldsymbol{ \alpha}_{n}, \boldsymbol{\beta}} 
 = P^{{\bf K'}_{n}, \boldsymbol{ \alpha'}_{n}, \boldsymbol{\beta}} $.
 In fact, all we have to show is that, after the first folding, exchanging an essential and 
 an inessential folding gives the same probability. After the first folding
 the probabilities are symmetric; therefore, assuming that $P$ is already a folding,
 we have
 \begin{eqnarray}
 P^{(K_1, \emptyset), (\alpha_1, \alpha),  \boldsymbol{\beta}}(\omega)&=&
 \frac{1}{Z_2}
 P^{K_1, \alpha_1,  \boldsymbol{\beta}}( \omega)
 P^{K_1, \alpha_1,  \boldsymbol{\beta}}( \overline \omega) \nonumber \\
 &=&  \frac{1}{Z_2}
 P^{K_1, \alpha_1,  \boldsymbol{\beta}}( \omega)^2
  \nonumber \\
 &=& \frac{1}{Z_2} ( \frac{1}{Z_1} P(\alpha_1 \omega)
 P(\alpha_1 \overline \omega) )^2 \nonumber \\
 &=&
\frac{( P(\alpha_1 \omega))^2
 (P(\alpha_1 \overline \omega) )^2}{\sum_{\omega'} ( P(\alpha_1 \omega'))^2
 (P(\alpha_1 \overline \omega') )^2}  \nonumber \\
 &=&
\frac{\frac{( P(\alpha_1 \omega))^2}{\sum_{\omega''} ( P(\alpha_1 \omega''))^2}
 \frac{(P(\alpha_1 \overline \omega) )^2}{\sum_{\omega''} ( P(\alpha_1 \omega''))^2}
    }{\sum_{\omega'}   \frac{P(\alpha_1 \omega'))^2}{\sum_{\omega''} ( P(\alpha_1 \omega''))^2}
 \frac{ (P(\alpha_1 \overline \omega') )^2}{\sum_{\omega''} ( P(\alpha_1 \omega''))^2}}  \nonumber \\
 &=&\frac{
 P^{\emptyset, \alpha}(\alpha_1 \omega) P^{\emptyset, \alpha}(\alpha_1 \overline \omega) 
    }{\sum_{\omega'} P^{\emptyset, \alpha}(\alpha_1 \omega') P^{\emptyset, \alpha}(\alpha_1 \overline \omega')  }  \nonumber \\
 &=&P^{( \emptyset,K_1), ( \alpha,\alpha_1),  \boldsymbol{\beta}}(\omega) \nonumber
 \end{eqnarray} 
 
\end{proof}
We restrict from now on to the essential tree of foldings. Note, in particular, 
that the essential tree of foldings has at most $(2^{|\Lambda|} |\Omega_{\Lambda}|)^{|\Lambda|}
< \infty$ infinite branches. In addition, in each branch there are at most
$L =|\Lambda | $ initial essential foldings, and then all the others are inessential.

\subsection{RCR's in the tree of foldings} 
Given $\Lambda, \mathcal B \subseteq \mathcal P(\lambda)$, and a probability $P$ on 
$ \mathcal P(\lambda)$, various foldings of $P$ can admit a $\mathcal B$-RCR. Some relevant
cases are as follows.
\begin{defi} \label{2.3}
We say that $P$ has a {\bf finite tree $\mathcal B$-RCR},
 if there exists a finite subtree $T$ of the
tree of foldings such that for each leaf 
$({\bf K}, \boldsymbol{ \alpha}, \boldsymbol{\beta})$ of $T$
there exists a $\mathcal B$-RCR $ \nu_{{\bf K}, \boldsymbol{ \alpha}, \boldsymbol{\beta}}$
of $ P^{{\bf K}, \boldsymbol{ \alpha}, \boldsymbol{\beta}}$.

For $\epsilon>0$, we say that $P$ has a {\bf finite tree $\epsilon$-$\mathcal B$-RCR },
 if there exists a finite subtree $T$ of the
tree of foldings such that for each leaf $({\bf K}, \boldsymbol{ \alpha}, \boldsymbol{\beta})$ of
$T$
there exists a $\epsilon$-$\mathcal B$-RCR $ \nu_{{\bf K}, \boldsymbol{ \alpha}, \boldsymbol{\beta}}$
of $ P^{{\bf K}, \boldsymbol{ \alpha}, \boldsymbol{\beta}}$.

We say that $P$ has an {\bf   $\mathcal B$-RCR at infinity} if 
for each
infinite branch in the essential tree of branches
 the limiting distribution $P^{\infty}$ in that branch has a 
 $\mathcal B$-RCR.
 
 We say that $P$ has an {\bf  asymptotic
  $\mathcal B$-RCR} if there exists an integer valued function
  $n_{\epsilon}$ such that $\lim_{\epsilon \rightarrow 0} 2^{n_{\epsilon}} \epsilon =0$
for which the following happens: 
for  every $\epsilon >0$,
$P$ has a finite tree $\epsilon$-$\mathcal B$-RCR on
the  subtree $T_{\epsilon}$ formed by the essential tree of foldings
  truncated after $n_{\epsilon}$ generations.

\end{defi}

\begin{lemma} \label{5.1.4}
If $P$ has a $\mathcal B$-RCR $\nu_{\infty}$ at infinity, then $\nu_{\infty}$
is also an
 asymptotic
  $\mathcal B$-RCR .
 
 \end{lemma}
 
 \begin{proof}
 Let $P$ have a $\mathcal B$-RCR at infinity, and consider $\epsilon >0$.
 Recall that in each infinite branch there are at most $L$  essential 
 foldings at the beginning, and choose $n=n_{\epsilon}$ the smallest integer such that
 $$
 a^{2^n}= |  \Omega_{\Lambda} | \frac{ \max_{\omega'': P^{\infty}(\omega')=0} 
(P^{{\bf K}_L, \boldsymbol{ \alpha}_L, \boldsymbol{\beta}}(\omega'') )^{2^{n-L}} }
{(P^{{\bf K}_L, \boldsymbol{ \alpha}_L, \boldsymbol{\beta}}(\omega') )^{2^{n-L}}}
\leq \epsilon$$
for each $\omega$ with $P^{\infty}(\omega) >0$,
where $a$ is taken as in Lemma  \ref{5.1.3}.
Since $P^{\infty}$ has a $\mathcal B$-RCR by assumption, and it approximates
$P^{{\bf K}_L, \boldsymbol{ \alpha}_L, \boldsymbol{\beta}}$ by less than
$a^{2^n} < \epsilon$ by  Lemma \ref{5.1.3}, it follows that
$P^{{\bf K}_L, \boldsymbol{ \alpha}_L, \boldsymbol{\beta}}$ has a
$\epsilon$-$\mathcal B$-RCR using $T=T_{\epsilon}$ as finite subtree.
In addition, as $n_{\epsilon}$ is chosen to be the smallest integer,
$\epsilon \approx a^{2^{n_{\epsilon}}}$, and 
$\lim_{\epsilon \rightarrow 0} 2^{n_{\epsilon}} \epsilon 
= \lim_{\epsilon \rightarrow 0}2^{n_{\epsilon}} a^{2^{n_{\epsilon}}}
=0$ as $a < 1$ from Lemma  \ref{5.1.3}, so that $P$ 
has an asymptotic
  $\mathcal B$-RCR .
 
 \end{proof} 
 
\section{FKG theory}

\subsection{Positive association}

A probability $P$ is positively associated  (PA) if for all increasing
$A, B \subseteq \Omega_{\Lambda}$, $P(A \cap B) \geq P(A) P(B)$.
Equivalently, for all increasing
$A$ and decreasing $B $, $P(A \cap B) \leq P(A) P(B)$.

We are going to find a very simple sufficient condition for positive 
association in terms of RCR. To this purpose, 
given a finite set $\Lambda$ and ordered sets $F_i$, we denote by
$\hat \omega_{\Lambda}$ the configuration such that 
$(\hat \omega_{\Lambda})_i = \max \{s: s \in F_i\}$.

If $|b| \leq 2$ for all $b \in \mathcal B(\Lambda)$ and the $F_i$'s are
all ordered sets, then a $\mathcal B$-RCR $\nu$ is {\bf ferromagnetic} if
for $b=\{i,j\}$,
$\hat \omega_b \in \eta_b$ for all $\eta_b \neq \emptyset$. Recall that if $|F_i|=2$ we say that $\nu$ is symmetric 
if $\omega_b \in \eta_b$ implies $\overline \omega_b \in \eta_b$.
We can assume that $F_i=\{0,1\}$, in case after relabeling of the elements of $F_i$; in such case,
a symmetric and ferromagnetic $\mathcal B$-RCR with $|b| \leq 2$ for
all $b$ is such that either $\eta_b = \Omega_b$ or $\eta_b = \{\{0,0\},\{1,1\}\}$,
analogously to what happens in the original FK representation (see Section 3.2).

\begin{theorem} \label{5.2.x}
Let  $F_i$ be ordered and $P$ be a probability on $\Omega_{\Lambda}$.
Suppose that there is a finite subtree $T$ of the tree of foldings such 
that each leaf of $T$ has a symmetric ferromagnetic $\mathcal B$-RCR
with $|b| \leq 2$ for all $b \in \mathcal B(\Lambda)$; then $P$ is PA.
\end{theorem}
\begin{proof} For each leaf 
$ ({\bf K}, \boldsymbol{ \alpha}, \boldsymbol{\beta})$,
$ P^{{\bf K}, \boldsymbol{ \alpha}, \boldsymbol{\beta}}$ is defined
on $\prod_i \{ \beta_i(1), \beta_i(2)\}$, which can be relabeled 
into $F_i=\{0,1\}$. The proof is divided in various steps.

(I) For each leaf 
$ ({\bf K}, \boldsymbol{ \alpha}, \boldsymbol{\beta})$ there is a 
$\mathcal B$-RCR $\nu_{{\bf K}, \boldsymbol{ \alpha}, \boldsymbol{\beta}}$
which is symmetric and ferromagnetic. If $\nu(\eta)>0$,
$C(\eta)$ is one of the clusters formed by active hyperbonds in $\eta$,
and $\omega \in \Omega_{ (\cup_{i=1}^n K_i)^c}$,
then either $C(\eta) \subseteq \omega^{-1}(1)$ or 
$C(\eta) \subseteq \omega^{-1}(0)$ as a bond $b=\{i,j\}$ with
$\omega_i \neq \omega_j$ can never be active.

Let $\Psi$ be a selection rule which only selects pairs of sets $(M, N)$
with $M \subseteq \omega^{-1}(1)$ and $N \subseteq \omega^{-1}(0)$; 
under this selection rule the event $A$ can only be recognized by $1$'s
and the event $B$ by $0$'s. Therefore, at most one of $C(\eta)\cap M \neq \emptyset$
or $C(\eta)\cap N \neq \emptyset$ is possible, hence $\Psi$ uses disjoint
clusters of $\eta$. Lemma \ref{2.3.2} with $\epsilon =0$ gives  that for all increasing
$A$ and decreasing $B$ we have that for each leaf
\begin{eqnarray} \label{5.2.1}
P^{{\bf K}, \boldsymbol{ \alpha}, \boldsymbol{\beta}} (A\cap B)
= P^{{\bf K}, \boldsymbol{ \alpha}, \boldsymbol{\beta}}(A \nb B)
\leq P^{{\bf K}, \boldsymbol{ \alpha}, \boldsymbol{\beta}}(A \cap \overline B).
\end{eqnarray}

(II) 
Let $ ({\bf K}, \boldsymbol{ \alpha}, \boldsymbol{\beta})$ be a node
of the tree of foldings, and assume that for its direct descendant 
\eqref{5.2.1} holds with the appropriate symbol replacements
to indicate the descendant. It descendants
are also obtained from a folding, so Lemma \ref{2.3.3} applies with  
$\Lambda =(\cup_{i=1}^n K_i)^c $ and $\epsilon =0$ to see that for
 $A$ increasing, $B$ decreasing in $\Omega_{ (\cup_{i=1}^n K_i)^c}$,
 for each node of the tree of foldings
 whose direct descendants satisfy \eqref{5.2.1} we have
 \begin{eqnarray} \label{5.2.2}
P^{{\bf K}, \boldsymbol{ \alpha}, \boldsymbol{\beta}} (A\cap B)
= P^{{\bf K}, \boldsymbol{ \alpha}, \boldsymbol{\beta}}(A \nb B)
\leq P^{{\bf K}, \boldsymbol{ \alpha}, \boldsymbol{\beta}}(A )
P^{{\bf K}, \boldsymbol{ \alpha}, \boldsymbol{\beta}}( B).
\end{eqnarray}
This gives the required inequality, but for the node $ ({\bf K}, \boldsymbol{ \alpha}, \boldsymbol{\beta})$. 
To get the result for $P$ we have to backtrack a little and bootstrap things.

(III) Suppose that for a node 
$ ({\bf K}, \boldsymbol{ \alpha}, \boldsymbol{\beta})$ 
of the tree of foldings all of its direct descendant satisfy
\eqref{5.2.1}. If $B$ is decreasing then $\overline B$ is increasing, and 
$(\overline B)^c$ is again decreasing. Then it follows from \eqref{5.2.2}
that 
 \begin{eqnarray}
P^{{\bf K}, \boldsymbol{ \alpha}, \boldsymbol{\beta}} (A\cap (\overline B)^c)
\leq P^{{\bf K}, \boldsymbol{ \alpha}, \boldsymbol{\beta}}(A )
P^{{\bf K}, \boldsymbol{ \alpha}, \boldsymbol{\beta}}((\overline B)^c),
\end{eqnarray}
which yields
\begin{eqnarray} \label{5.2.33}
P^{{\bf K}, \boldsymbol{ \alpha}, \boldsymbol{\beta}} (A\cap \overline B)
\geq P^{{\bf K}, \boldsymbol{ \alpha}, \boldsymbol{\beta}}(A )
P^{{\bf K}, \boldsymbol{ \alpha}, \boldsymbol{\beta}}(\overline B).
\end{eqnarray}
By symmetry of each folding after the first one, for all $B$, 
$P^{{\bf K}, \boldsymbol{ \alpha}, \boldsymbol{\beta}}(\overline B)
=P^{{\bf K}, \boldsymbol{ \alpha}, \boldsymbol{\beta}}( B)$. Then,
for all increasing $A$ and decreasing $B$ we have that for a vertex 
$ ({\bf K}, \boldsymbol{ \alpha}, \boldsymbol{\beta})$ 
of the tree of foldings other than the root, all of  whose direct descendant satisfy
\eqref{5.2.1} it holds
 \begin{eqnarray} \label{5.2.4}
P^{{\bf K}, \boldsymbol{ \alpha}, \boldsymbol{\beta}} (A\cap B)
&\leq& P^{{\bf K}, \boldsymbol{ \alpha}, \boldsymbol{\beta}}(A )
P^{{\bf K}, \boldsymbol{ \alpha}, \boldsymbol{\beta}}( B) \\
&=&P^{{\bf K}, \boldsymbol{ \alpha}, \boldsymbol{\beta}}(A )
P^{{\bf K}, \boldsymbol{ \alpha}, \boldsymbol{\beta}}( \overline B)
\leq  P^{{\bf K}, \boldsymbol{ \alpha}, \boldsymbol{\beta}}(A \cap \overline B),
\nonumber
\end{eqnarray}
which is again \eqref{5.2.1}, but now brought up one step.

(IV) As  \eqref{5.2.1} holds for all the leaves, and if it holds for all descendants
it holds also for the parent node, provided this is different  from the root,
then  \eqref{5.2.1}  holds for all nodes, including thus all descendants of 
the root $\bf 0$.

(V) The final step is for the original $P$. By (IV), we can apply Lemma \ref{2.3.3}
to get that for increasing $A$ and decreasing $B$,
$P(A\cap B) = P(A \nb B) \leq P(A) P(B)$
so that $P$ is PA as required.
\end{proof}

\begin{lemma} \label{5.2.2}
Let  $F_i$ be ordered and $P$ be a probability on $\Omega_{\Lambda}$.
Suppose $P$ has an asymptotic $\mathcal B$-RCR such that 
  each leaf of $T_{\epsilon}$ 
 has a symmetric ferromagnetic $\epsilon$-$\mathcal B$-RCR 
with $|b| \leq 2$ for all $b \in \mathcal B(\Lambda)$; then $P$ is PA.

 \end{lemma}
 \begin{proof}  
 The proof is a repetition of the proof of Theorem \ref{5.2.x} with some modifications.
 Applying Lemma \ref{2.3.2} with the error term, \eqref{5.2.1} is replaced by
 \begin{eqnarray} \label{5.2.5}
P^{{\bf K}, \boldsymbol{ \alpha}, \boldsymbol{\beta}} (A\cap B)
= P^{{\bf K}, \boldsymbol{ \alpha}, \boldsymbol{\beta}}(A \nb B)
\leq P^{{\bf K}, \boldsymbol{ \alpha}, \boldsymbol{\beta}}(A \cap \overline B)
\pm \epsilon 2^{|\Lambda|+1}.
\end{eqnarray}
Then \eqref{5.2.33} is replaced by
\begin{eqnarray} \label{5.2.66}
P^{{\bf K}, \boldsymbol{ \alpha}, \boldsymbol{\beta}} (A\cap \overline B)
\geq P^{{\bf K}, \boldsymbol{ \alpha}, \boldsymbol{\beta}}(A )
P^{{\bf K}, \boldsymbol{ \alpha}, \boldsymbol{\beta}}(\overline B) \pm \epsilon 2^{|\Lambda|+1}
\end{eqnarray}
and \eqref{5.2.4} is replaced by
 \begin{eqnarray} \label{5.2.77}
P^{{\bf K}, \boldsymbol{ \alpha}, \boldsymbol{\beta}} (A\cap B)
\leq P^{{\bf K}, \boldsymbol{ \alpha}, \boldsymbol{\beta}}(A \cap \overline B)
\pm 2  \epsilon 2^{|\Lambda|+1}.
\end{eqnarray}
Notice the extra factor of $2$ due to the possible summation of the errors in 
\eqref{5.2.66} and \eqref{5.2.77}.

Following Part (IV) above, since $T_{\epsilon} $ has 
$n_{\epsilon}$ generations we have that the direct descendants of the root
$\bf 0$ satisfy
 \begin{eqnarray} \label{5.2.8}
P^{{\bf K}, \boldsymbol{ \alpha}, \boldsymbol{\beta}} (A\cap B)
\leq  P^{{\bf K}, \boldsymbol{ \alpha}, \boldsymbol{\beta}}(A \cap \overline B)
+ 2^{n_{\epsilon}} \epsilon 2^{|\Lambda|+1},
\nonumber
\end{eqnarray}
and following Part (V) $P(A\cap B)  \leq P(A) P(B) + 2^{n_{\epsilon}} \epsilon 2^{|\Lambda|+1}.$
By definition of asymptotic representability, $\lim_{\epsilon \rightarrow 0} 2^{n_{\epsilon}} \epsilon
=0$, so that $P(A\cap B)  \leq P(A) P(B)$.

 \end{proof}
 Combining Lemmas \ref{5.1.4} and \ref{5.2.2} 
 we get the final result of this section:
\begin{theorem} \label{5.2.3}
Let  $F_i$ be ordered and $P$ be a probability on $\Omega_{\Lambda}$.
If $P$ has a  $\mathcal B$-RCR at infinity which is ferromagnetic and
symmetric, with $|b| \leq 2$ for all $b \in \mathcal B(\Lambda)$, then $P$ is PA.
\end{theorem}

\subsection{FKG theorem} 
The FKG theorem, which we are going to show at the end of this section,
states that if a probability on a lattice satisfies the FKG condition then
it is PA.
We say that $P$ {\bf satisfies the FKG condition} if 
for all $\omega, \omega' \in \Omega_{\Lambda}=
\prod_{i \in \Lambda} F_i$, $F_i$ ordered sets, $|F_i| \leq 2$ 
\begin{eqnarray} \label{2.2.2}
P(\omega \vee \omega')P(\omega \wedge \omega') 
\geq P(\omega) P(\omega')
\end{eqnarray}
Recall that $\hat \omega_i= \max_{x \in F_i} x$. The FKG condition
takes a nice form in the foldings.
\begin{lemma} \label{2.2.1}
$P$ satisfies the FKG condition if and only if for
every $K \subseteq \Lambda, \alpha \in \Omega_K, 
\beta(1), \beta(2) \in \Omega_{K^c}, \beta_i(1) \neq \beta_i(2),
\omega_{K^c} \in \Omega_{K^c}$
\begin{eqnarray} \label{2.2.3}
P^{\fe}(\omega_{K^c}) 
\leq P^{\fe}(\hat \omega_{K^c}).
\end{eqnarray}

\end{lemma}
\begin{proof} 
If \eqref{2.2.2} is satisfied, then for all 
$K, \alpha, \beta(1), \beta(2), \omega_{K^c}$ 
let $\omega = \alpha \hat \omega_{K^c}$
and $\omega'= \alpha \overline{ \omega_{K^c}}$. We have
 \begin{eqnarray} \label{2.2.4}
 \omega=\alpha \omega_{K^c}, 
 \quad \omega'=\alpha \overline { \omega_{K^c}},
\quad  \omega \vee \omega'=\alpha \hat \omega_{K^c},
\quad  \omega \wedge \omega'= 
\alpha \overline{\hat \omega_{K^c}}.
\end{eqnarray}
Thus 
\begin{eqnarray} \label{2.2.5}
P^{\fe}( \omega_{K^c})
 &=& P(\alpha \omega_{K^c})P( \alpha \overline{ \omega_{K^c}})
=P(\omega) P(\omega')
 \nonumber  \\
&\leq&
  P(\omega \vee \omega')P(\omega \wedge \omega') 
   \nonumber  \\
&=&
P(\alpha \hat \omega_{K^c})P( \alpha  \overline{\hat \omega_{K^c}}) 
=P^{\fe}( \hat \omega_{K^c}) .
\end{eqnarray}

Viceversa,
if $P$ satisfies \eqref{2.2.3} and  $\omega, \omega' \in \Omega_{\Lambda}$,
let $K$ be the set of $i \in \Lambda$ in which the two configurations 
$\omega, \omega'$ are equal; let also
$\alpha_i = \omega_i$ for $i \in K$, $\beta_i(1)=\omega_i$,
 $\beta_i(2)=\omega'_i$ for $i \in K^c$, $(\omega_{K^c})_i
 = \omega_i$ for $i \in K^c$.
 Then
  \eqref{2.2.4}
still holds and 
 the \eqref{2.2.5}
can be used in reversed order to show that \eqref{2.2.2}  holds.
\end{proof}

\begin{lemma}
$P$ satisfies the FKG condition if and only if for
all foldings $(\fe)$, $P^{\fe} $ satisfies the FKG condition as well.
\end{lemma}
\begin{proof}
For all $(\fe)$, and for all $u \in K^c$, $\omega(1), \omega(2) \in \Omega_{K^c}$
\begin{eqnarray} 
\overline{ \omega_u(1) \vee \omega_u(2)} &=& 
\overline{ \omega_u(1) }\wedge \overline{ \omega_u(2)} \nonumber \\
\overline{ \omega_u(1) \wedge \omega_u(2)} &=& 
\overline{ \omega_u(1) }\vee \overline{ \omega_u(2)} . \nonumber
\end{eqnarray}
If $P$ is FKG then
\begin{eqnarray} 
&& P^{\fe}(  \omega(1) \vee \omega(2))P^{\fe}(  \omega(1) \wedge \omega(2)) \nonumber  \\
 && \quad =
 P(\alpha (  \omega(1) \vee \omega(2)) )
 P(\alpha \overline {(  \omega(1) \vee \omega(2))} )
  P(\alpha (  \omega(1) \wedge \omega(2)) )
 P(\alpha \overline {(  \omega(1) \wedge \omega(2))} ) \nonumber  \\
 && \quad =
 P(\alpha (  \omega(1) \vee \omega(2)) )
 P(\alpha ( \overline {  \omega(1)} \wedge \overline {  \omega(2)} ) )
  P(\alpha (  \omega(1) \wedge \omega(2)) )
 P(\alpha (\overline {  \omega(1)} \vee \overline {\omega(2)} ) )
 \nonumber  \\
&&\quad \geq P(\alpha   \omega(1)) P(\alpha \omega(2)) 
 P(\alpha  \overline{ \omega(1)}) P(\alpha \overline{ \omega(2)}) 
 \nonumber  \\
&&\quad = P^{\fe}(  \omega(1) )P^{\fe}(   \omega(2)) 
   \nonumber .
\end{eqnarray}
Viceversa, if for all $(\fe)$, $P^{\fe}$ satisfies the FKG condition,
then for $\omega \in \Omega_{K^c}$, and, by symmetry of the foldings
\begin{eqnarray} 
(P^{\fe}( \hat \omega))^2
 &=& P^{\fe}( \omega \vee \overline{\omega} )P^{\fe}( \omega \wedge \overline{\omega} )
 \nonumber  \\
&\geq&
  P^{\fe}( \omega  )P^{\fe}( \overline{\omega} )= (P^{\fe}(  \omega))^2
   \nonumber  .
\end{eqnarray}
Hence $P$ is FKG by Lemma \ref{2.2.1}.

\end{proof}
It follows that if $P$ satisfies the FKG conditions then so do all its iterated foldings, 
and the probability $P^{\infty} $ in the limit of each infinite branch. Each of the
$P^{\infty} $ is a two-valued, symmetric distribution uniform on some subset $D \subseteq
\Omega_{(\cup_{n \in \mathbb N})^c}$; we see now that all probabilities
of this form which satisfy the FKG condition are products of independent clusters, so that
they have a ferromagnetic and symmetric $\mathcal B$-RCR made of pairs.
\begin{theorem} \label{5.3.3}
If $P$ is two-valued, symmetric, uniform on $D \subseteq
\Omega_{\Lambda}=\prod_{u \in \Lambda}\{0,1\}$ for some set $\Lambda$, and satisfies the FKG
condition, then $P$ has a ferromagnetic and symmetric $\mathcal B$-RCR
such that $|b| \leq 2$ for all $b \in \mathcal B$.
\end{theorem}
Note  that  the reverse also holds, but we omit  the proof
for brevity,
 
\begin{proof}
If $P$ is as in the thesis, then consider $\mathcal B
=\{ b: b=\{u,v\}, u,v \in \Lambda\}$, and  let $\bar \eta \in 
\prod_{b \in \mathcal B} \Omega_b$ be defined by
$$
\bar \eta_{\{u,v\}}=
\begin{cases}
\{(1,1),(0,0)\} \quad \text{ if for all } \omega \in D, \omega_u = \omega_v \\
\Omega_b \quad \text{ if there exists } \omega \in D \text{ such that } \omega_u \neq \omega_v
\end{cases}
$$
Clearly, if $\omega \in D$ and $\bar \eta_{\{u,v\}} = \{(1,1),(0,0)\}$
then $\omega_u=\omega_v$; thus, $\omega \sim \bar \eta$. We then let
$\nu := \delta_{\bar \eta}$; clearly, $\nu$ is ferromagnetic and symmetric. We claim 
that it is a $\mathcal B$-RCR of $P$.

For all $\omega \in D$, $P(\omega)=\frac{1}{Z} \sum_{\eta \sim \omega} \nu(\eta)
=\frac{\nu(\bar \eta)}{Z} 
= \frac{1}{Z}$;  if we show that for $\omega \notin D$, $\sum_{\eta \sim \omega} \nu(\eta)=0$,
which is to say $\eta \not\sim \omega$,
then indeed $Z=|D|$ and $P$ would be uniform on $D$, which is the thesis. 
It is then sufficient to prove that if $\omega$ is constant on the $\bar \eta$ clusters,
then $\omega \in D$.

Let $\Omega(\bar \eta) =\{\omega: \omega \text{ is constant on the } \bar \eta \text{ clusters } \}
$; clearly, $D \subseteq \Omega(\bar \eta)$ and we want to show that equality holds.
Let $C_1(\bar \eta), \dots, C_m(\bar \eta)$ be the cluster defined by the
active bonds of $\bar \eta$, and let
$t:\Omega(\bar \eta) \rightarrow \{0,1\}^m$ be defined by $(t(\omega))_j=\omega_u$
for $u \in C_j(\bar \eta)$. Observe that $t(D)=\{\tau \in \{0,1\}^m:
\text{there exists } \omega  \in \Omega(\bar \eta) \text{ with } t(\omega)=\tau \}$
satisfies
\begin{enumerate}
\item $t(D)$ is a sublattice of $ \{0,1\}^m$: in fact, for $\tau(1), \tau(2) \in t(D)$,
$\tau(1)=t(\omega(1)) , \tau(2)=t(\omega(2)), \omega(1), \omega(2) \in D$,
we have $P(\omega(1))=P(\omega(2))\neq 0$, hence, by the FKG property
 of $P$, $P(\omega(1) \vee \omega(2))=P(\omega(1) \wedge \omega(2))
\neq 0$;
therefore, $\tau(1) \vee \tau(2) = t(\omega(1) \vee \omega(2) ) \in t(D)$
and $\tau(1) \wedge \tau(2) = t(\omega(1) \wedge \omega(2) ) \in t(D)$;
\item $t(D)$ is symmetric: by the symmetry of $P$, if $\tau = t(\omega)$
then $P(\omega)>0$ which implies $P(\overline {\omega})>0$,
hence $\overline \omega \in D$, thus $\overline \tau =t( \overline \omega)
\in t(D)$;
\item $t(D)$ separates points of $M=\{1, \dots, m\}$
where $m$ is the number of clusters defined by active bonds in 
$\bar \eta$, in the sense that for all
$i,j \in M, i \neq j$, there is $\tau^{i,j} \in t(D)$ such that 
$\tau^{i,j}_i \neq \tau^{i,j}_j$: in fact, by definition of $\bar \eta$, if $i \neq j$
and $u \in C_i(\bar \eta), v \in C_j(\bar \eta)$, then there exists $\omega \in D$
such that $\omega_u \neq \omega_v$; hence, $(t(\omega))_i 
=\omega_u \neq \omega_v=(t(\omega))_j$. Therefore, 
$\tau^{i,j} = t(\omega)$ separates $i$ and $j$. Here we have used that
 $\mathcal B
=\{ b: b=\{u,v\}, u,v \in \Lambda\}$.

\end{enumerate}
The next lemma shows that a symmetric sublattice which separates points
coincides with the whole lattice; hence, $t(D) = \{0,1\}^m$, and 
every configuration which is constant on the clusters of $\bar \eta$
is in $D$, which concludes the proof.
\end{proof}

\begin{lemma}
Consider $T=\{0,1\}^m$. If a sublattice $L$ of $T$ is symmetric and 
separates ponts, then $L=T$.

\end{lemma}
\begin{proof}
Since $L$ separates points, $L\neq \emptyset$; let $\tau \in L$,
then $\overline \tau \in L$ by symmetry, and $\hat \tau= \tau \vee \overline \tau
\in L$, as it is a lattice, and $\overline{\hat \tau} \in L$ by symmetry.

For $\tau \in T$, let $m(\tau)=|\{ i =1, \dots, m: \tau_i=0\}|$.
If $m(\tau)=0, m$ then $\tau= \hat \tau$ or  $\tau=\overline{ \hat \tau}$, 
so in any case $\tau \in L$. If $m(\tau)=1$ then let $i $ be such that
$\tau_i=0$; for $j \neq i$ let $\tau^{(j)} \in L$ be such that 
separates $i$ and $j$, which is $\tau^{(j)}_i=0, \tau^{(j)}_j=1$. Then
$\tau = \vee_{j \neq i} \tau^{(j)} \in L$.

Finally, let $\overline m = \min \{m': \text{ there exists }
\tau \in T\setminus L \text{ such that }m(\tau)=m'\}$. For $\tau 
\in T \setminus L$ such that $m(\tau)= \overline m$
and $i$ such that $\tau_i=0$, let $\tau^{(i)}$ be such that
$$
(\tau^{(i)})_j=
\begin{cases}
\tau_j  \quad \quad \text{ if } j \neq i \\
1  \quad \quad \text{ if } j = i.
\end{cases}
$$
Then $m(\tau^{(i)})=m(\tau)-1 < \overline m$ and 
$\tau^{(i)} \in L$ by induction. Moreover, $\tau= \wedge_{i: \tau_i=0} \tau^{(i)}$.
Hence, $\tau \in L$, which is a contradiction.

\end{proof}
Our main conclusion
\begin{theorem} \label{5.3.5} (FKG theorem)
If $P$ satisfies the FKG condition then $P$ is PA.
\end{theorem}
\begin{proof}
If $P$ satisfies the FKG condition, then so does the limit of the foldings
down each infinite branch. The limit satisfies then the conditions of 
Theorem \ref{5.3.3} and hence it has a  ferromagnetic and symmetric $\mathcal B$-RCR
such that $|b| \leq 2$ for all $b \in \mathcal B$.
But then Theorem \ref{5.2.3} implies that $P$ is PA.

\end{proof}

\section{Some theory of negative association} 

\subsection{Negative association}
We assume that $\Omega_{\Lambda} = \prod_{i \in \Lambda}
F_i$, with $F_i$ finite and ordered. 
We say that two events $A, B \in \Omega_{\Lambda}$ have
{\bf disjoint support} if there exists $N \subseteq \Lambda$
(not dependent on a configuration)
such that for all $\omega \in A\cap B$,
$[\omega]_N \subseteq A$ and $[\omega]_{N^c} \subseteq B$.
\begin{defi}
$P$ is negatively associated (NA) if for all $A, B \subseteq \Omega_{\Lambda}$,
increasing and with disjoint support, 
$P(A \cap B) \leq P(A) P(B)$

\end{defi}
The theory of negative association is more difficult than that of positive association,
see \cite{P00, BrJ11, BBL09, Br07, DJR07, DR98, KN10, M09}. We develop a version here, 
which, for the most part, reproduces, with 
a completely different strategy which mirrors the one obtained for
positive association, the results in 
 \cite{P00}. The additional difficulty with respect to our 
 FKG theory can be seen from the fact  that a RCR  sufficient to guarantee NA is the mirror
 image of the one for PA, but, in addition, needs to be concentrated on 
 isolated edges.

Let $\mathcal B(\Lambda)$ be such that $|b| \leq 2$ for all $b \in \mathcal B(\Lambda)$,
 let $F_i=\{0,1\}$ for all $i \in \Lambda$, then we say that 
 a $\mathcal B$-RCR $\nu$ is symmetric and {\bf antiferromagnetic} if
for all $b$, $\eta_b=\{ \{(0,1),(1,0) \},\Omega_b \}$.

We say that $\nu$ is {\bf  concentrated on isolated edges}
if $\nu$ is a $\mathcal B$-RCR with $|b| \leq 2$ for all $b \in \mathcal B(\Lambda)$,
and, moreover, 
 if for all $b(1), b(2)
\in \mathcal B$ such that $\eta_{b(1)}=\eta_{b(2)}\neq \Omega_b$
implies $b(1) \cap b(2) =\emptyset$; in other words, active bonds
are isolated with $\nu$ probability one. In \cite{BG13} 
various inequalities of BK type are proven by RCR of foldings of 
a probability  $P$ which are symmetric, antiferromagnetic and
concentrated on isolated edges. We now develop a theory a negative
association using such representation for the limit probability of the trees of foldings.

\begin{theorem} \label{5.4.3}
Let  $F_i$ be ordered and $P$ be a probability on $\Omega_{\Lambda}$.
If $P$ has a  $\mathcal B$-RCR at infinity which is antiferromagnetic,
symmetric and
concentrated on isolated edges,  then $P$ is NA.
\end{theorem}
\begin{proof}
The proof can be obtained by following the proofs of Theorems \ref{5.2.x}
and \ref{5.2.3} with suitable adaptations.

In particular, it is proven in \cite{BG13} that if all foldings of a
probability $P$ has a $\mathcal B$-RCR  which is antiferromagnetic,
symmetric and
concentrated on isolated edges, with $|b| \leq 2$ for all $b \in \mathcal B(\Lambda)$, then $P$
satisfies the following: for all $A, B$ increasing and with disjoint support, 
$P(A \cap B) \leq P(A \cap \overline B)$. This inequality is approximate
if the representation is only approximate, as one would get
from just having a $\mathcal B$-RCR at infinity.

Part (III) of Theorem \ref{5.2.x} can be followed without modifications, 
by just noting that if $A$ and $B$ have disjoint support, then also
$B^c$ and $\overline B$ have disjoint support from $A$, as seen by
just using the same set $N$.

All the remaining parts of the proofs work without relevant modifications.

\end{proof}

\subsection{Negative FKG theory} 
For a configuration $\omega \in \Omega_{\Lambda} = \{0,1\}$
we indicate $|\omega|=\sum_{i \in \Lambda} \omega_1$ the number of $1$'s
in $\omega$.
\begin{defi}
We say that a probability $P$ satisfies the negative FKG condition, or it is NFKG,
if for every  folding $(K, \alpha)$ it happens that if $\omega_{K^c}
\in \Omega_{K^c}$ satisfies
\begin{eqnarray} \label{5.5.1}
| |\omega_{K^c} | - \frac{|K^c|}{2} | &\leq& 1/2 
\end{eqnarray}
then $P^{K, \alpha} (\omega_{K^c}) \geq 
P^{K, \alpha} (\omega_{K^c}')$ for all 
$\omega_{K^c}' \in 
\Omega_{B_i \cap K^c}$.

\end{defi}
Note that if $\omega_{K^c}$ satisfies  \eqref{5.5.1}, then also  $\overline{\omega_{K^c}}$
does.
With these definitions we want to mirror Theorem \ref{5.3.5}. But this is actually
easier done if we
start from a more restricted .\begin{defi}
We say that a probability $P$ satisfies the strict negative FKG condition, or it is SNFKG,
if for every  folding $(K, \alpha)$ it happens that if $\omega_{K^c}
\in \Omega_{K^c}$ satisfies \eqref{5.5.1} then
\begin{enumerate}
\item $P^{K, \alpha} (\omega_{K^c}) = 
P^{K, \alpha} (\omega_{K^c}')$ for all 
$\omega_{K^c}' \in 
\Omega_{B_i \cap K^c}$ which also satisfy \eqref{5.5.1};
\item if $\omega_{K^c}' $ does not satisfy \eqref{5.5.1},
then $P^{K, \alpha} (\omega_{K^c}) >
P^{K, \alpha} (\omega_{K^c}')$.
\end{enumerate}

\end{defi}
Of course, SNFKG implies NFKG.
\begin{lemma} \label{5.5.2}
IF $P$ is SNFKG  then for each folding
$(K, alpha)$, $P^{K,\alpha}$ is SNFKG.
\end{lemma}
\begin{proof} Let $P$ be SNFKG, and consider $K_1, K_2 \subseteq K_1^c,
\alpha_1 \in \Omega_{K_1^c}, \alpha_2 \in \Omega_{(K_1 \cup K_2)^c}$.
Then we consider the folding $P^{(K_1, K_2), (\alpha_1, \alpha_2)}$ of
$P^{K_1, \alpha_1}$, and two configurations
 $\omega', \tilde{\omega} \in \Omega_{(K_1 \cup K_2)^c}$ such
 that $\omega'$ satisfies \eqref{5.5.1}; 
we have
\begin{eqnarray} 
 P^{(K_1, K_2), (\alpha_1, \alpha_2)}(\omega')&=&
 \frac{1}{Z_2 Z_1^2}
 P(\alpha_1 \alpha_2\omega') P(\alpha_1 \overline{\alpha_2} \overline{\omega'})
 P(\alpha_1 \alpha_2\overline{\omega'})
 P(\alpha_1 \overline{\alpha_2} \omega')
  \nonumber  \\
&\geq & \frac{1}{Z_2 Z_1^2}
 P(\alpha_1 \alpha_2 \tilde \omega) P(\alpha_1 \overline{\alpha_2} \overline{\tilde \omega})
 P(\alpha_1 \alpha_2\overline{\tilde \omega})
 P(\alpha_1 \overline{\alpha_2} \tilde \omega) \nonumber  \\
&= &P^{(K_1, K_2), (\alpha_1, \alpha_2)}(\tilde \omega)
  \nonumber 
  \end{eqnarray}
 as for the folding $(K_1 \cup K_2,\alpha_1\alpha_2)$ 
 \begin{eqnarray} 
 P(\alpha_1 \alpha_2\omega')P(\alpha_1 \alpha_2\overline{\omega'})
 &=&P^{(K_1\cup K_2), (\alpha_1 \alpha_2)}(\omega')
 \nonumber  \\
&\geq &P^{(K_1\cup K_2), (\alpha_1 \alpha_2)}(\tilde \omega)
 \nonumber  \\
&= &P(\alpha_1 \alpha_2 \tilde \omega) P(\alpha_1 \alpha_2\overline{\tilde \omega}),
  \nonumber 
  \end{eqnarray}
 and analogously for the folding $(K_1 \cup K_2,\alpha_1\overline \alpha_2)$,
 with equality if also $\tilde \omega$ satisfies \eqref{5.5.1},
 and strict inequalities otherwise.
 
 \end{proof}

\begin{lemma}  \label{5.5.3}
If $P$  is SNFKG then $P$ has a  
an antiferromagnetic and symmetric $\mathcal B$-RCR at infinity
 concentrated on isolated edges.
\end{lemma}
\begin{proof}
By Definition \ref{5.1.5} we consider the uniform probability $P^{\infty}$
on the maxima of a leaf $({\bf K}, \boldsymbol{ \alpha})$ of the essential tree
of foldings. As $P$ is SNFKG then each leaf is SNFKG;
then, the maxima are configurations $\omega$ satisfying
\eqref{5.5.1} with $K=\cup K_i$ with $K_i$ the sets of the
essential foldings leading to the leaf.

Let $\nu$ be the uniform distribution on the complete pairings of $\Lambda$,
i.e. configurations $\eta$'s whose active bonds are disjoint but pair every two  vertices
(see \cite{BG13} for more details), with the exception of at most one vertex.

We need to show that  each 
 $\omega$ satisfying \eqref{5.5.1}  is compatibile with the same number of  such
 complete pairings, and that no other $\omega$ is compatible.
 
 In fact, recall that the active bonds are antiferromagnetic, so each carries exactly one
 value $1$: a compatible configuration has thus half of the vertices being one,
 if $|\Lambda|$ is even, and in any case it must satisfy \eqref{5.5.1}.
 
 On the other hand, each $\omega $ satisfying \eqref{5.5.1} is compatible
 with exactly $\lceil \frac{|K^c|}{2} \rceil !$ complete pairings, and this
 finishes the proof.
 
\end{proof}
\begin{corollary} \label{5.5.10}
If $P$ is SNFKG then it is NA.
\end{corollary}
\begin{proof}
By Lemma \ref{5.5.3}, $P$ has a $\mathcal B$-RCR
by an antiferromagnetic and symmetric RCR concentrated on isolated edges.
By Theorem \ref{5.4.3} this implies that $P$ is NA.

\end{proof}
\begin{theorem}
If $P$ is NFKG then it is NA.
\end{theorem}
\begin{proof}
We prove the result by a perturbation method.

Let $P$ be NFKG, let $A,B \subseteq \Omega_{\Lambda}$
be increasing events, and suppose there exists $N \subseteq \Lambda$ such that
for all $\omega \in A \cap B$, $[\omega]_N \subseteq A, 
[\omega]_{N^c} \subseteq B$.

Given $\epsilon >0$ consider the probability such that 
$$
P_{\epsilon} (\omega) = \frac{P(\omega) }{Z_{\epsilon}}
(1+\epsilon)^{  |\{ b=\{u,v\} \in \mathcal B(\Lambda): \, \omega_u \neq \omega_v\}|
}.
$$
If $P$ is NFKG then $P_{\epsilon}$ is SNFKG. Suppose, in fact, that $\omega \in \Omega_{\Lambda}$ satisfies
 \eqref{5.5.1}.
 
1. If also $\omega'\in \Omega_{\Lambda}$ satisfies \eqref{5.5.1},
then $P(\omega) = P(\omega')$, by definition of 
NFKG, and
 \begin{eqnarray} \label{5.5.5}
P_{\epsilon} (\omega) = \frac{P(\omega) }{Z_{\epsilon}}
(1+\epsilon)^{\lfloor\frac{|\Lambda|}{2}\rfloor
} =\frac{P(\omega') }{Z_{\epsilon}}
(1+\epsilon)^{\lfloor\frac{|\Lambda|}{2}\rfloor
} =P_{\epsilon} (\omega') 
\end{eqnarray}

2. If, on the other hand,  $\omega' \in \Omega_{\Lambda}$ does not
satisfy \eqref{5.5.1}, then 
$P(\omega) \geq P(\omega')$, and, by definition,
it is enough to show that 
\begin{eqnarray}
  \lfloor\frac{|\Lambda|}{2}\rfloor
>  |\{ b=\{u,v\} \in \mathcal B(\Lambda), \omega'_u \neq \omega'_v\}|.
\nonumber
\end{eqnarray}
For this, it is enough that for all $m$, $\Lambda=\{1, \dots,m\}, 
\Omega=\{0,1\}^{\Lambda}, \epsilon >0$, the function 
$\omega \rightarrow |\{ b=\{u,v\} \in \mathcal B(\Lambda), u, v \in \Lambda,\omega_u \neq \omega_v\}|$
has its maximum for the $\omega$'s such that $||\omega|-\frac{m}{2} | \leq \frac{1}{2}$
and it is strictly smaller for all $\omega$'s such that $||\omega|-\frac{m}{2} | >\frac{1}{2}$.
But, in fact, if $|\omega| = k$ then
\begin{eqnarray*}
&&|\{ b=\{u,v\} \in \mathcal B(\Lambda), u, v \in \Lambda,\omega_u \neq \omega_v\}| \\
&&  \quad \quad  \quad \quad =k(m-k)-(\frac{k(k-1)}{2} + \frac{(m-k)(m-k-1)}{2}) \\
&&  \quad \quad  \quad \quad = 3 k m - 3k^2-m^2+m=:a_k.
\end{eqnarray*}

Now, $a_{k+1}-a_k=3m-3-6k$ and $a_{k+1} \geq a_k$ is equivalent to 
$k\leq \frac{m-1}{2}$ with equality if and only if $|k-\frac{m}{2}| \leq \frac{1}{2}$;
in such case $a_k=\lfloor \frac{m}{2} \rfloor$, which is what is required to show that
$P_{\epsilon}$ is SNFKG.

Since $P_{\epsilon}$ is SNFKG than it is NA by Corollary \ref{5.5.10}. Hence,
$P_{\epsilon} (A \cap B) \leq P_{\epsilon}(A) P_{\epsilon}(B)$
for every $\epsilon >0$. As $\Lambda$ and $F_i$'s
are finite, for every $A$ $\lim_{\epsilon \rightarrow 0} P_{\epsilon}(A) = P(A)$,
so also $P$ is NA.
\end{proof}

\begin{example}
If $P$ is exchangeable, then let $p_k=P(\omega)$ if $|\omega|=k$, 
$k=0, 1, \dots, |\Lambda|$. Then if $P$ is NFKG one can deduce that
for $s<m/2$
$$
p_{r+m-s} p_{r+s} \geq p_{r+m-s+1} p_{r+s-1};
$$
 from this it is easy to see that $P$ is NFKG
if and only if
$p_{k+1} p_{k-1} \leq p_k^2$ for $K=1, \dots, |\Lambda|-1$, 
i.e. $P$ is ultra log concave (ULC). This gives an alternative proof of the 
result in \cite{P00} that an exchangeable ULC is NA (actually,
we prove that $P$ is CNA$+$, in the terminology of \cite{P00}, as also done there).
\end{example}


\end{document}